\documentclass[12pt]{amsart}
\usepackage{amsmath,amssymb,amsbsy,amsfonts,latexsym,amsopn,amstext,cite,
                                               amsxtra,euscript,amscd,bm,mathabx}
\usepackage{url}
\usepackage[colorlinks,linkcolor=blue,anchorcolor=blue,citecolor=blue,backref=page]{hyperref}
\usepackage{color}
\usepackage{graphics,epsfig}
\usepackage{graphicx}
\usepackage{float} 
\usepackage[english]{babel} 
\usepackage{mathtools}
\usepackage{todonotes}
\usepackage{url}
\usepackage[colorlinks,linkcolor=blue,anchorcolor=blue,citecolor=blue,backref=page]{hyperref}

\usepackage[norefs,nocites]{refcheck}

\hypersetup{breaklinks=true}

\usepackage[norefs,nocites]{refcheck}
\usepackage[english]{babel}
\begin{document}

\newtheorem{thm}{Theorem}
\newtheorem{lem}[thm]{Lemma}
\newtheorem{claim}[thm]{Claim}
\newtheorem{cor}[thm]{Corollary}
\newtheorem{prop}[thm]{Proposition} 
\newtheorem{definition}[thm]{Definition}
\newtheorem{question}[thm]{Open Question}
\newtheorem{conj}[thm]{Conjecture}
\newtheorem{prob}{Problem}

\theoremstyle{remark}
\newtheorem{rem}[thm]{Remark}

\newcommand{\GL}{\operatorname{GL}}
\newcommand{\SL}{\operatorname{SL}}
\newcommand{\lcm}{\operatorname{lcm}}
\newcommand{\ord}{\operatorname{ord}}
\newcommand{\Op}{\operatorname{Op}}
\newcommand{\Tr}{\operatorname{Tr}}
\newcommand{\Nm}{\operatorname{Nm}}

\numberwithin{equation}{section}
\numberwithin{thm}{section}
\numberwithin{table}{section}

\def\vol {{\mathrm{vol\,}}}
\def\squareforqed{\hbox{\rlap{$\sqcap$}$\sqcup$}}
\def\qed{\ifmmode\squareforqed\else{\unskip\nobreak\hfil
\penalty50\hskip1em\null\nobreak\hfil\squareforqed
\parfillskip=0pt\finalhyphendemerits=0\endgraf}\fi}

\def \balpha{\bm{\alpha}}
\def \bbeta{\bm{\beta}}
\def \bgamma{\bm{\gamma}}
\def \blambda{\bm{\lambda}}
\def \bchi{\bm{\chi}}
\def \bphi{\bm{\varphi}}
\def \bpsi{\bm{\psi}}
\def \bomega{\bm{\omega}}
\def \btheta{\bm{\vartheta}}

\def\eps{\varepsilon}

\newcommand{\bfxi}{{\boldsymbol{\xi}}}
\newcommand{\bfrho}{{\boldsymbol{\rho}}}

\def\Kab{\sfK_\psi(a,b)}
\def\Kuv{\sfK_\psi(u,v)}
\def\SaUV{\cS_\psi(\balpha;\cU,\cV)}
\def\SaAV{\cS_\psi(\balpha;\cA,\cV)}

\def\SUV{\cS_\psi(\cU,\cV)}
\def\SAB{\cS_\psi(\cA,\cB)}

\def\Kmnp{\sfK_p(m,n)}

\def\KKap{\cH_p(a)}
\def\KKaq{\cH_q(a)}
\def\KKmnp{\cH_p(m,n)}
\def\KKmnq{\cH_q(m,n)}

\def\Klmnp{\sfK_p(\ell, m,n)}
\def\Klmnq{\sfK_q(\ell, m,n)}

\def \SALMNq {\cS_q(\balpha;\cL,\cI,\cJ)}
\def \SALMNp {\cS_p(\balpha;\cL,\cI,\cJ)}

\def \SACXMQX {\fS(\balpha,\bzeta, \bxi; M,Q,X)}

\def\SAMJp{\cS_p(\balpha;\cM,\cJ)}
\def\SAMJq{\cS_q(\balpha;\cM,\cJ)}
\def\SAqMJq{\cS_q(\balpha_q;\cM,\cJ)}
\def\SAJq{\cS_q(\balpha;\cJ)}
\def\SAqJq{\cS_q(\balpha_q;\cJ)}
\def\SAIJp{\cS_p(\balpha;\cI,\cJ)}
\def\SAIJq{\cS_q(\balpha;\cI,\cJ)}

\def\RIJp{\cR_p(\cI,\cJ)}
\def\RIJq{\cR_q(\cI,\cJ)}

\def\TWXJp{\cT_p(\bomega;\cX,\cJ)}
\def\TWXJq{\cT_q(\bomega;\cX,\cJ)}
\def\TWpXJp{\cT_p(\bomega_p;\cX,\cJ)}
\def\TWqXJq{\cT_q(\bomega_q;\cX,\cJ)}
\def\TWJq{\cT_q(\bomega;\cJ)}
\def\TWqJq{\cT_q(\bomega_q;\cJ)}

 \def \xbar{\overline x}
  \def \ybar{\overline y}

\def\cA{{\mathcal A}}
\def\cB{{\mathcal B}}
\def\cC{{\mathcal C}}
\def\cD{{\mathcal D}}
\def\cE{{\mathcal E}}
\def\cF{{\mathcal F}}
\def\cG{{\mathcal G}}
\def\cH{{\mathcal H}}
\def\cI{{\mathcal I}}
\def\cJ{{\mathcal J}}
\def\cK{{\mathcal K}}
\def\cL{{\mathcal L}}
\def\cM{{\mathcal M}}
\def\cN{{\mathcal N}}
\def\cO{{\mathcal O}}
\def\cP{{\mathcal P}}
\def\cQ{{\mathcal Q}}
\def\cR{{\mathcal R}}
\def\cS{{\mathcal S}}
\def\cT{{\mathcal T}}
\def\cU{{\mathcal U}}
\def\cV{{\mathcal V}}
\def\cW{{\mathcal W}}
\def\cX{{\mathcal X}}
\def\cY{{\mathcal Y}}
\def\cZ{{\mathcal Z}}
\def\Ker{{\mathrm{Ker}}}

\def\NmQR{N(m;Q,R)}
\def\VmQR{\cV(m;Q,R)}

\def\Xm{\cX_m}

\def \A {{\mathbb A}}
\def \B {{\mathbb A}}
\def \C {{\mathbb C}}
\def \F {{\mathbb F}}
\def \G {{\mathbb G}}
\def \L {{\mathbb L}}
\def \K {{\mathbb K}}
\def \PP {{\mathbb P}}
\def \Q {{\mathbb Q}}
\def \R {{\mathbb R}}
\def \Z {{\mathbb Z}}
\def \fS{\mathfrak S}

\def\e{{\mathbf{\,e}}}
\def\ep{{\mathbf{\,e}}_p}
\def\eq{{\mathbf{\,e}}_q}

\def\\{\cr}
\def\({\left(}
\def\){\right)}
\def\fl#1{\left\lfloor#1\right\rfloor}
\def\rf#1{\left\lceil#1\right\rceil}

\def\Tr{{\mathrm{Tr}}}
\def\Nm{{\mathrm{Nm}}}
\def\Im{{\mathrm{Im}}}

\def \oF {\overline \F}

\newcommand{\pfrac}[2]{{\left(\frac{#1}{#2}\right)}}

\def \Prob{{\mathrm {}}}
\def\e{\mathbf{e}}
\def\ep{{\mathbf{\,e}}_p}
\def\epp{{\mathbf{\,e}}_{p^2}}
\def\em{{\mathbf{\,e}}_m}

\def\Res{\mathrm{Res}}
\def\Orb{\mathrm{Orb}}

\def\vec#1{\mathbf{#1}}
\def \va{\vec{a}}
\def \vb{\vec{b}}
\def \vm{\vec{m}}
\def \vu{\vec{u}}
\def \vv{\vec{v}}
\def \vx{\vec{x}}
\def \vy{\vec{y}}
\def \vz{\vec{z}}
\def\flp#1{{\left\langle#1\right\rangle}_p}
\def\T {\mathsf {T}}

\def\sfG {\mathsf {G}}
\def\sfK {\mathsf {K}}

\def\mand{\qquad\mbox{and}\qquad}

\title[Exponential Sums and  Power Generator]
{Exponential Sums with Sparse Polynomials and Distribution of the Power Generator}

\author{Subham Bhakta}
\address{School of Mathematics and Statistics, University of New South Wales, Sydney, NSW 2052, Australia.} 
\email{subham.bhakta@unsw.edu.au}

\author{Igor E. Shparlinski}
\address{School of Mathematics and Statistics, University of New South Wales, Sydney, NSW 2052, Australia.} 
\email{igor.shparlinski@unsw.edu.au}

\begin{abstract}   We obtain new bounds on complete rational exponential sums  with 
sparse polynomials modulo a prime, under some mild conditions on the degrees of the monomials 
of  such polynomials.
These  bounds, when they  apply,  give  explicit versions of a result of J.~Bourgain (2005). In turn, as an application,  we also obtain an explicit version of a result of J.~Bourgain (2010)
on rational   exponential sums  with sparse polynomials modulo an arbitrary composite number. We then use  one of these bounds to study the multidimensional distribution of the classical power generator
of pseudorandom numbers,  which has not been possible within previously known results. 
\end{abstract}



\maketitle
\tableofcontents
%
%
%

\section{Introduction}
\subsection{Set-up and motivation}
\label{sec:setup}

Let $q$ be a positive integer. Given a polynomial $f \in \Z_q[X]$ over  the residue ring
we define the following exponential sums 
\begin{equation}
\label{eq: Exp Sum}
S_q(f) = \sum_{x \in \Z_q^*} \eq(f(x)),
\end{equation}
over the group of units  $\Z_q^*$, where 
\[
\eq(z) = \exp(2 \pi i z/q),
\]
and we always assume that the elements of $\Z_q$ are represented by the set $ \{0, \ldots, q-1\}$.

Here we are mostly interested in  {\it sparse  polynomials\/}, that is, polynomials of the form 
\begin{equation}
\label{eq: Sparse Poly}
 f(X)=\sum_{i=1}^{r} a_iX^{e_i}  \in \Z_q[X], 
\end{equation}
 with a small (compared to the degree) number  $r$ of distinct exponents $e_r > \ldots >e_1 \ge 1$,
 for which we always assume 
 \begin{equation}
\label{eq: GCD}
\gcd(a_1, \ldots, a_r,q)=1
\end{equation}
  
We first discuss the case of prime $q=p$, where we use  $\F_p=\Z_p$ to denote the finite field of $p$ elements.

The celebrated result of Weil~\cite{Weil} asserts that for any non-constant 
polynomial $f \in \F_p[X]$,  over  a finite field $\F_p=\Z_p$ of $p$ elements, where $p$ is prime,   we have
\begin{equation}
\label{eq:Weil}
\left| S_p(f) \right| \le (d - 1) p^{1/2},
\end{equation}
where $d = \deg f$, 
see also, for example,~\cite[Chapter~11]{IwKow} and~\cite[Chapter~6]{Li1}. 

Clearly the bound~\eqref{eq:Weil} becomes trivial for $d \ge p^{1/2}+1$, which is 
the limit of our present knowledge. However, in the case of  sparse 
polynomials~\eqref{eq: Sparse Poly} where we can also assume  $p-1\ge e_r > \ldots >e_1 \ge 1$
 and $a_1, \ldots, a_r \in \F_p^*$, 
 better bounds are known. 
 
 Most of attention was dedicated to the monomial case, that is,  to $r=1$. Since the first improvement 
 of the general bound~\eqref{eq:Weil} given in~\cite{Shp1}, several  stronger bounds have been 
 obtained, see, for example~\cite{HBK,Kon, Shkr} and references therein. We observe, that  if $r=1$, 
 then without loss of generality, we can assume that  $d = e_1\mid p-1$, which reduces the corresponding 
 sums to estimating exponential sums of elements of subgroups of $\F_p^*$ of order $(p-1)/e$.
 In particular, by a result of Shkredov~\cite[Theorem~1]{Shkr}, for $f(X) = aX^d$ with $a \in \F_p^*$, we have
 \begin{equation}
\label{eq:Shkr}
S_p(f) \ll d^{1/2} p^{2/3} (\log p)^{1/6},
\end{equation}
where, as usual, the notations $U = O(V)$, $U \ll V$ and $ V\gg U$  
are equivalent to $|U|\le c V$ for some positive constant $c$, 
which throughout this work may depend only on $r$ (and
thus the implied constant  in~\eqref{eq:Shkr} is absolute). 
Note that in some ranges, stronger bounds are available, but~\eqref{eq:Shkr}
is quite sufficient for our purpose. 

For $r=2$ a large variety of bounds is available~\cite{GRS, OSV1, OSV2, ShpVol, ShpWang}, which require  rather strong restrictions on some greatest common divisors associated with the powers $d$ and $e$ of monomials 
in   $f(X) = aX^d + bX^e \in \F_p[X]$. For our purpose we use a result of 
 Cochrane and  Pinner~\cite[Theorem~1.3]{CoPin2}:
 \begin{equation}
\label{eq:CoPin}
|S_p(f)|\le \gcd(d-e, p-1) + 2.292 g^{13/46} p^{89/92}
\end{equation}
where $g = \gcd(d,e,p-1)$.

For larger values of $r$ there are also some estimates that go beyond the Weil bound~\eqref{eq:Weil}, 
see~\cite{CoPin1, MPSS, ShpWang}. However, they apply only to rather special choices  
of the exponents $e_1, \ldots, e_r$.  

The case of arbitrary polynomials $f$ as in~\eqref{eq: Sparse Poly}, under some  natural condition 
(which is essentially necessary as well) has been treated by Bourgain~\cite{Bourg1}. 
Namely, for arbitrary $r \ge 2$ and arbitrary polynomials $f$ as in~\eqref{eq: Sparse Poly}, under a natural condition that 
\[
\gcd(e_i, p-1), \, \gcd(e_i-e_{j}, p-1) \le p^{1-\varepsilon}, \qquad i, j =1, \ldots, r, \ i \ne j,
\]
for some fixed  $\varepsilon>0$, Bourgain~\cite[Theorem~1]{Bourg1},  using methods of additive combinatorics, 
 has obtained the bound 
\begin{equation}
\label{eq:Bourg}
\left| S_p(f) \right| \le  p^{1-\delta},
\end{equation}
with some $\delta>0$ depending only on $r$ and $\varepsilon>0$. Unfortunately, Bourgain~\cite{Bourg1}
has  not provided any explicit formulas for $\delta$ as a function of $r$ and $\varepsilon>0$, while 
an attempt to do so in~\cite[Theorems~4 and~5]{Pop} seems to be faulty. 
In particular, it appears that the argument in~\cite{Pop}  misquotes the result of~\cite[Corollary~16]{Shkr3}, which,  
after correcting, leads to exponentially smaller saving (and yet the rest of~\cite{Pop} 
is still to be verified).

For composite $q$ and ``dense'' polynomials 
$f(X)=a_dX^d + \ldots +a_1X \in \Z_q[X]$ of degree $d$ with $\gcd(a_1, \ldots, a_r,q)=1$, we have 
\[
S_{q}(f) \ll q^{1-1/d}, 
\]
see, for example,~\cite{DiQi,Stech}. Note that this is a slight improvement of the 
classical bound of Hua (containing an additional factor $q^{o(1)}$), see, for example,~\cite[Theorem~7.1]{Vau}.
For sparse polynomials as in~\eqref{eq: Sparse Poly} and~\eqref{eq: GCD}, but of {\it fixed\/}  degree $\deg f = e_r$ 
 a better bound of the type 
 \begin{equation}
\label{eq:Shp q}
|S_{q}(f)| \le q^{1-1/r+o(1)}
\end{equation}
is given in~\cite[Theorem~1]{Shp2}. 

Bourgain~\cite{Bourg2}, using the bound~\eqref{eq:Bourg}, has given several estimates
for the sums $S_{q}(f)$ with sparse polynomials  as in~\eqref{eq: Sparse Poly} and~\eqref{eq: GCD}, 
 For example, by~\cite[Theorem~1]{Bourg2}, 
for any $\eta> 0$ there exist  $\delta>0$ depending only on $r$ and $\eta>0$, such that for 
 \begin{equation}
\label{eq:Bourg q}
\left| S_q(f) \right| \le  q^{1-\delta},
\end{equation}
provided that 
\[
q \ge \exp\(d^\eta\), 
\]
where $d = \deg f = e_r$. In particular, the derivation of~\eqref{eq:Bourg q} in~\cite{Bourg2}
is based on~\eqref{eq:Bourg} from~\cite{Bourg2}. Here, despite that  our explicit 
version of~\eqref{eq:Bourg} is more restrictive in terms of the conditions 
on the exponents $e_1, \ldots, e_r$, it is  sufficient to get a version of~\eqref{eq:Bourg q} 
with an explicit value of $\delta$. A further improvement comes from the new way we 
treat  exponential sums modulo $p^m$, when $m \ge 2$ is not too large,  see Lemma~\ref{lem:pm} below.

 \subsection{Applications} 
 We also recall that the bound~\eqref{eq:CoPin} has been used in~\cite{OstShp}
 to study the uniformity of distribution properties of the  
{ \it power generator\/} of pseudorandom numbers,
More precisely,  given an integer  $e\geq 2$ and a prime  $p$, this generator outputs   the sequence $\{u_n\}$, defined by
\begin{equation}
\label{eq:Pow}
u_n \equiv u_{n-1}^e \pmod p, \quad 0 \le u_n \le  p-1, \qquad n = 1, 2,
\ldots\,,
\end{equation}
with the initial value $u_0 \not \equiv 0 \pmod p$. 
Since the  power generator has extensive  applications in cryptography,  and thus has been investigated in 
a number of works, see, for example,~\cite{Bourg1,ChouShp,FHS1,FHS2,FrKoSh,FrSh,FPS,Kon,KurPom,ShaHu} and the references therein, 
we also refer to surveys~\cite{TopWin, Win}.

We use our new bounds on exponential sums with sparse polynomials to study  simultaneous distribution 
of $s$-tuples of the  power generator, while the previous results allow explicit bounds only in the 
one-dimensional case as in~\cite{OstShp}. We also discuss the introduced in~\cite{OstShp} 
multivariate 
generalisations of the power generator to which our approach applies as well.

 \section{Main results} 
 
  \subsection{Exponential sums with sparse polynomials over $\F_p$}
Here we use the bounds~\eqref{eq:Shkr} and~\eqref{eq:CoPin} as the basis for the 
inductive argument of Bourgain~\cite{Bourg1}, in the form outlined by 
Garaev~\cite[Section~4.4]{Gar}  to derive a bound on exponential sums $S_p(f)$ 
with sparse polynomials $f$ as in~\eqref{eq: Sparse Poly}.
 
For a real $\varepsilon>0$, we set 
\begin{equation}
\label{eq:kappa2}
\kappa_1(\varepsilon) = \kappa_2(\varepsilon) =  \varepsilon,
 \end{equation}
 and define the sequence $\kappa_{r}(\varepsilon)$, $r =3,4, \ldots$,
 recursively as follows
\begin{equation}
\label{eq:kappar}
 \kappa_r(\varepsilon) =  \frac{\varepsilon}{t_r(\varepsilon)},
 \end{equation} 
 where
 \begin{equation}
\label{eq:tr-eps}
t_r(\varepsilon) =\rf{\frac{r-2- \varepsilon}{ 2\kappa_{r-1}(\varepsilon)}}+2.
  \end{equation}

\begin{thm}\label{thm:anyrbetter}
Let $f$ be as in~\eqref{eq: Sparse Poly} with $a_1, \ldots, a_r \in \F_p^*$. Assume that there exists $\varepsilon$ with $3/92\ge \varepsilon>0$ and such that for all  $1 \le i\neq j \le r$:
    \begin{align*}
&\gcd(e_i,p-1) \le   0.5 p^{7/26 +14\varepsilon/13}, \\
&  \gcd(e_i -e_j,p-1) \le   p^{8/13 +32\varepsilon/13}, \\
& \gcd(e_i,e_j, p-1) \le  p^{3/26-46\varepsilon/13} . 
\end{align*}  
    Then,   we have
 \[
 S_p(f) \ll p^{1 -  \kappa_r(\varepsilon)}.
\]
\end{thm}
 
\begin{rem}There could be several variants of Theorem~\ref{thm:anyrbetter}
with slightly relaxed conditions on the corresponding greatest common divisors but with weaker 
estimates. Since, generically, the greatest common divisor of two integers tends to be small, 
we imposed the above more stringent conditions, which however maximise the strength of the bound. 
\end{rem}

Furthermore, we also have a 
simpler version of Theorem~\ref{thm:anyrbetternew} with $\varepsilon= 3/92$, restricting to $\gcd(e_i,p-1) =1$, 
while relaxing the condition on $ \gcd(e_i -e_j,p-1)$.

It is convenient to define 
 \begin{equation}
\label{eq:rho-r}
\rho_r =  \kappa_r(3/92) .
  \end{equation}

\begin{thm}\label{thm:anyrbetternew}
Let $f$ be as in~\eqref{eq: Sparse Poly} with $a_1, \ldots, a_r \in \F_p^*$. Assume that for all  $1 \le i\neq j \le r$:
\[
\gcd(e_i,p-1)=1 \mand \gcd(e_i -e_j,p-1) \le   p^{16/23}.
\]
Then,  we have 
\[
S_p(f) \ll p^{1 -  \rho_r}.
\]
\end{thm}

\begin{rem}
Using Theorems~\ref{thm:anyrbetter} and~\ref{thm:anyrbetternew} combined with the Koksma--Sz\"usz 
inequality, see  Lemma~\ref{lem:K-S}, one can get  explicit versions of a result of Bombieri, 
Bourgain and Konyagin~\cite[Theorem~17]{BBK}, see also~\cite[Corollary~9]{BKS}, counting the 
number of solutions to the system of congruences
\[
a_i x^{e_i} \equiv y_i \pmod p, \qquad x \in \F_p, \ y_i \in [A_i+1, A_i + H_i], \ i =1, \ldots, r,
\]
with some integers $A_i$ and $1\le H_i \le p$. 
\end{rem}

  \subsection{Exponential sums with sparse polynomials over $\Z_q$}  
We merge the bound of Theorem~\ref{thm:anyrbetternew} with the argument 
Bourgain~\cite{Bourg2} to derive the following explicit version of~\eqref{eq:Bourg q} .

We now define  
 \begin{equation}
\label{eq:sigma-r}
\sigma_r =  \min\{\kappa_r(3/184),1/(50r^3)\} .
  \end{equation}
In fact $\sigma_r <   \kappa_r(3/184)$ only for $r = 2$, that is,
 alternatively we can define
 \[
 \sigma_r = \begin{cases} 
  \kappa_r(3/184),& \text{for} \ r=1~\mathrm{and}~r\ge 3,\\
  1/400,& \text{for} \ r =  2.
  \end{cases}
\]

\begin{thm}\label{thm:coprimeq}
Let $q \ge 2$ be an integer and let $f$ be as ~\eqref{eq: Sparse Poly} and~\eqref{eq: GCD}
of degree $d$. Then uniformly over $d$, we have
\[
S_q(f)  \ll \exp\(d^{18}\) q^{1 - \sigma_r+o(1)},  \qquad \text{as}\ q \to \infty, 
\]
where $\sigma_r$ is given by~\eqref{eq:sigma-r}. 
\end{thm}

\begin{rem} We emphasise that the main point of Theorem~\ref{thm:coprimeq} is that the 
degree $d$ may grow with $q$, as the saving $\sigma_r$ and  the implied constant depend only on 
the sparsity $r$ rather than on $d$, while  the rate of decay of $o(1)$ in the exponent of $q$ 
depends only on $q$. Clearly, for a fixed $d$,  the bound~\eqref{eq:Shp q} of~\cite[Theorem~1]{Shp2}
is much stronger. 
\end{rem}

\begin{rem} The power $18$ of $d$ in  Theorem~\ref{thm:coprimeq} is not optimised
and can easily be reduced and the cost of redefining and reducing $\sigma_r$ as well. We have also ignored the factor $1/50$ at the front of $d^{18}$ which our argument gives, see 
the final estimate in the proof of   Theorem~\ref{thm:coprimeq}.
\end{rem}

\begin{rem}
The limit of our technique with respect to the dependence on $d$   is $d^{26/3+\varepsilon}$, with any $\varepsilon > 0$,  instead of $d^{18}$ in Theorem~\ref{thm:coprimeq} due to~\eqref{eq: cond 4}. However, if we impose that the exponents $e_i$ are pairwise coprime, this issue disappears, leaving~\eqref{eq: cond 2} as the  bottleneck. 
Thus,  tightening~\eqref{eq: cond 2} and~\eqref{eq: cond 3} according to the conditions in Theorem~\ref{thm:anyrbetter}, provides a much smaller exponent than $26/3$.
\end{rem}

\subsection{Multidimensional distribution of the power generator}
 \label{sec:mulf distr} 
  
We recall that the  {\it discrepancy\/}  $D(\Gamma)$ 
of a sequence $\Gamma=(\bgamma_n)_{n= 1}^N$ of $N$ points 
in the $s$-dimensional unit cube $[0,1]^s$, 
 is defined by
\[
 D(\Gamma) =\sup_{\cB \subseteq[0,1]^s}\left|\frac{N(\cB)}{N}- ( \beta_1-\alpha_1)   \cdots(\beta_s-\alpha_s)\right|,
\]
with the supremum is taken over all boxes  
\[
\cB = [\alpha_1, \beta_1] \times \ldots  \times [\alpha_s, \beta_s]\subseteq  [0,1)^s,
\]
and where  $N(\cB)$ is the number of elements of $\Gamma$, which fall in $\cB$, 
and 
\[ 
\vol \cB = ( \beta_1-\alpha_1)   \cdots(\beta_s-\alpha_s)
\] 
is the volume of $\cB$. 

Furthermore, for an integer $\vartheta \not \equiv 0 \pmod p$, we denote by $D_{e, \vartheta,s}(p,N)$ 
the discrepancy of the points 
\[\(\frac{u_{n}}{p}, \ldots, \frac{u_{n+s-1}}{p}\),  \qquad n =1, \ldots, N,
\]
 generated by~\eqref{eq:Pow} 
with $u_0 = \vartheta$. 

For $s=1$, a bound on $D_{e, \vartheta,1}(p,N)$ is given by~\cite[Corollary~3]{OstShp}. However, the 
approach of~\cite{OstShp} does not extend to $s\ge 2$. Here we use 
Theorem~\ref{thm:anyrbetternew} to establish such a result.

Clearly, if $\gcd(e,p-1)=1$ and $\gcd(\vartheta, p) = 1$,  then it is easy to see that 
the sequence~\eqref{eq:Pow} is purely periodic. 

\begin{thm}\label{thm:discr PowGen} Let $\gcd(\vartheta, p) = 1$ be of multiplicative order $T$ modulo $p$
and let  $\tau$  be 
the period of  the sequence~\eqref{eq:Pow}. Then,  for any fixed $s\geq 1$, for $N \le \tau$ we have the following estimate
\[
D_{e, \vartheta,s}(p,N) \le  N^{-1/2} p^{1/2 -  0.5  \rho_{2s}+o(1)}, 
\]
where $\rho_r$ is given by~\eqref{eq:rho-r}. 
\end{thm}

Obviously, the bound of Theorem~\ref{thm:discr PowGen}   is nontrivial for $\tau \ge N \ge  p^{1-  \rho_{2s}+\varepsilon}$ with any fixed $\varepsilon > 0$.

\begin{rem}It is easy to see that  the bound of  Theorem~\ref{thm:discr PowGen} also holds for 
the discrepancy of the points
\[\(\frac{u_{n+h_1}}{p}, \ldots, \frac{u_{n+h_s}}{p}\),  \qquad n =1, \ldots, N,
\]
corresponding to arbitrary shifts $0 \le h_1 < \ldots < h_s < \tau$ satisfying $h_s - h_1 = p^{o(1)}$ as $p \to \infty$. 
\end{rem}

\section{Proofs of Theorems~\ref{thm:anyrbetter} and~\ref{thm:anyrbetternew}}

\subsection{Preliminaries} 
We use the inductive idea of Bourgain~\cite{Bourg1}, however we apply it in the form outlined by  
Garaev~\cite[Section~4.4]{Gar}, also applied 
in~\cite{BBG,OSV2}.  It is important to observe that the basis of induction is formed by  the cases $r=1$ and $r=2$
(just $r=1$ is not enough). 

We prove both Theorems~\ref{thm:anyrbetter} and~\ref{thm:anyrbetternew} simultaneously, giving 
more details for the proof of Theorem~\ref{thm:anyrbetter}, which is more technically involved. 

Let us consider the exponential sum $S_p(f)$ as in~\eqref{eq: Exp Sum} with polynomials $f$ as in~\eqref{eq: Sparse Poly}. 
We also use  $I_{r,t}$ to denote  the number of solutions to  the system of equations
 \begin{equation}\label{eq: Syst rt}
 x_1^{e_i} + \cdots + x_t^{e_i}  =  y_1^{e_i} + \cdots + y_t^{e_i},  \qquad i = 1, \ldots, r, 
\end{equation}
in  $x_1, y_1, \ldots, x_t, y_t \in \F_p^*$. 

\begin{lem}
\label{lem:S and I}    
For any integer $s,t \ge 1$ we have
\[
S_p(f) ^{2st} \ll  p^{2st-2t-2s+r}  I_{r,s}I_{r,t}.  
\]
\end{lem}

\begin{proof}
For any integer $t\geq 1$, we write
\begin{align*}
S_p(f)^{t} & = \sum_{x_1, \ldots, x_t\in  \F_p} \ep\(f\(x_1\) + \ldots + f\(x_t\)\)\\
&  = \frac{1}{p-1} \sum_{z   \in \F_p^*}
 \sum_{x_1, \ldots, x_t\in  \F_p}\ep\(f\(x_1 z \) + \ldots + f\(x_tz\)\)\\
& =   \frac{1}{p-1} 
 \sum_{x_1, \ldots, x_t\in  \F_p} \sum_{z   \in \F_p^*}   \e_p\(\sum_{i=1}^r a_i z^{e_i}\(x_1^{e_i}  + \cdots + x_{t}^{e_i} \)\) .
\end{align*}

Denoting by  $J_{r,t}(\lambda_1 ,\cdots, \lambda_{r})$ 
     the number of solutions to  the system of equations
\[
 a_i\(x_1^{e_i} + \cdots + x_t^{e_i}\) = \lambda_i, \qquad i = 1, \ldots, r, 
\]
in $x_1, \ldots, x_t \in \F_p^*$. 
we derive 
\begin{align*}
|S_p(f)|^{t}\ &\ll p^{-1}   \sum_{\lambda_1, \ldots, \lambda_r  \in \F_p} 
J_{r,t}(\lambda_1 ,\cdots, \lambda_{r}) \left| \sum_{z \in \F_p^{*}}
        \e_p\(\lambda_1 z^{e_1} +\cdots + \lambda_r z^{e_r} \) \right| .
  \end{align*}
Since $a_1, \ldots, a_r \ne 0$, we have 
 \begin{align*}
& \sum_{\lambda_1, \ldots, \lambda_r \in \F_p}   J_{r,t}(\lambda_1,  \ldots, \lambda_r) = (p-1)^t, \\
& \sum_{\lambda_1, \ldots, \lambda_r \in \F_p}   J_{r,t}(\lambda_1,  \ldots, \lambda_r)^2 = I_{r,t}. 
\end{align*}

Next, writing
\[
 J_{r,t}(\lambda_1,  \ldots, \lambda_r) =   J_{r,t}(\lambda_1,  \ldots, \lambda_r)^{(s-1)/s}  \(J_{r,t}(\lambda_1,  \ldots, \lambda_r)^2\)^{1/(2s)}
 \]
 and 
applying the H{\"o}lder inequality,  we derive
\begin{align*}
\left|S_p(f) \right|^{2st}&\ll p^{-2s}  \( \sum_{\lambda_1, \ldots, \lambda_r  \in \F_p} 
J_{r,t}(\lambda_1 ,\cdots, \lambda_{r})\)^{2s-2} \\
& \qquad   \ \qquad   \qquad  \sum_{\lambda_1, \ldots, \lambda_r \in \F_p}   J_{r,t}(\lambda_1,  \ldots, \lambda_r)^2 \\
& \qquad   \qquad  \qquad   \qquad   \sum_{\lambda_1, \ldots, \lambda_r \in \F_p}    \left| \sum_{z \in \F_p^{*}}
        \e_p\(\lambda_1 z^{e_1} +\cdots + \lambda_r z^{e_r} \) \right|^{2s}\\
        & \ll p^{2st-2s-2t}  I_{r,t}   \sum_{\lambda_1, \ldots, \lambda_r \in \F_p}    \left| \sum_{z \in \F_p^{*}}
        \e_p\(\lambda_1 z^{e_1} +\cdots + \lambda_r z^{e_r} \) \right|^{2s}. 
  \end{align*}
Since by the orthogonality of characters we have 
\[
   \sum_{\lambda_1, \ldots, \lambda_r \in \F_p}    \left| \sum_{z \in \F_p^{*}}
        \e_p\(\lambda_1 z^{e_1} +\cdots + \lambda_r z^{e_r} \) \right|^{2s}
        = p^{r}  I_{r,s}, 
\]
we conclude the proof.
\end{proof}

To apply Lemma~\ref{lem:S and I}  inductively, we use the trivial inequality $ I_{r,t}\le  I_{r-1,t}$, 
where $I_{\nu,t}$ is defined as $I_{r,t}$ by with respect to only the first $\nu\le r$ equations
in~\eqref{eq: Syst rt}. Hence, using the orthogonality of exponential functions, we write 
 \begin{equation}\label{eqn:r-1}
 \begin{split} 
I_{r-1,t}& = \frac{1}{p^{r-1}}  
              \sum_{\lambda_1, \ldots, \lambda_{r-1}\in \F_p}
              \left|\sum_{y\in  \mathbb{F}_p^{*}}  \e_p({\lambda_1 y^{e_1} + \cdots + \lambda_{r-1}y^{e_{r-1}}}) 
             \right | ^{2t} \\
& = p^{2t-r+1} +
              O\(\Delta\), 
\end{split} 
\end{equation}
with 
 \[
 \Delta =  \frac{1}{p^{r-1}} \sum_{\substack{\lambda_1, \ldots, \lambda_{r-1}\in \F_p\\
 (\lambda_1, \ldots, \lambda_{r-1}) \ne \mathbf 0 }}
\left|\sum_{y\in  \mathbb{F}_p^{*}}  \e_p({\lambda_1 y^{e_1} + \cdots + \lambda_{r-1}y^{e_{r-1}}}) 
\right | ^{2t}, 
\]
where $\mathbf 0$ is a zero vector in $\F_p^{r-1}$.  
Thus we estimate $ I_{r,t}$ via exponential sums with $(r-1)$-sparse polynomials, which enable inductive approach. 
Furthermore, for $t=2$ we use the following bound which is a weakened form or a result of 
Cochrane and Pinner~\cite[Theorem~7.1]{CoPin2}.

\begin{lem}\label{lem:I22}
For any  $\varepsilon>0$  we have
\[
I_{2,2}\ll p^{3-4\varepsilon},
\]
provided that  
\begin{align*}
&\gcd(e_1,p-1), \, \gcd(e_2,p-1)\le  0.5 p^{7/26 +14\varepsilon/13}, \\
&  \gcd(e_1 -e_2,p-1) \le  p^{8/13 +32\varepsilon/13}, \\
& \gcd(e_1,e_2, p-1) \le  p^{3/26-46\varepsilon/13}. 
\end{align*}  
\end{lem}

\begin{proof} Let $d =  \gcd(e_1,e_2, p-1)$. Assume that $p$ is sufficiently large. 
By~\cite[Theorem~7.1]{CoPin2}  we have the bound 
 \begin{equation}\label{eq: I22 CP-bound}
 I_{2,2}\ll  d^{26/23} p^{66/23},
 \end{equation}
provided that (with slightly relaxed constants) 
 \begin{equation}\label{eq: I22 CP-cond}
\begin{split}
& \gcd(e_1,p-1),  \gcd(e_1 -e_2,p-1) \le   (p/d)^{16/23}, \\
&  \gcd(e_2,p-1)\le  0.5 (p/d)^{7/23} .
\end{split} 
 \end{equation}
 We now easily verify that for $d \le p^{3/26-46\varepsilon/13}$ the conditions we impose on $e_1$ and $e_2$ are stricter than 
 required in~\eqref{eq: I22 CP-cond} (for example, we could only request $\gcd(e_1,p-1) \le p^{8/13 +32\varepsilon/13}$),  
 while the desired bound is weaker than~\eqref{eq: I22 CP-bound}.  
\end{proof} 
Using  the same argument as in the proof of Lemma~\ref{lem:I22}, while taking $d=1$, 
   we also see that~\cite[Theorem~7.1]{CoPin2} leads to  the following 
simpler version of Lemma~\ref{lem:I22}.

\begin{lem}\label{lem:newI22}
We have
\[I_{2,2}\ll p^{3-3/23},
\]
provided that  
\begin{align*}
&\gcd(e_1,p-1)=\gcd(e_2,p-1)=1, \\
&  \gcd(e_1 -e_2,p-1) \le  p^{16/23}.
\end{align*}  
\end{lem}

\subsection{Induction basis}  
 First we observe that the bounds~\eqref{eq:Shkr} and~\eqref{eq:CoPin} ensure that 
 under the condition of  Theorems~\ref{thm:anyrbetter} and~\ref{thm:anyrbetternew}  the desired results hold for $r=1$ and $r=2$. 
 
  Indeed, for $r=1$, under the assumption of $e_1$ in  Theorem~\ref{thm:anyrbetter}
 the bound~\eqref{eq:Shkr} implies 
 \[
 S_p(f) \ll d^{1/2} p^{2/3} (\log p)^{1/6} 
 \ll   \(p^{7/26 +14\varepsilon/13}\)^{1/2}  p^{2/3} (\log p)^{1/6}  \ll p^{1-\varepsilon}, 
\]
since for $\varepsilon < 31/240$ we have
\[
\frac{1}{2} \(7/26 +14\varepsilon/13\) + 2/3 < 1-\varepsilon. 
\]

Furthermore,  under the assumption of $e_1$ in  Theorem~\ref{thm:anyrbetternew} 
we obviously have $ S_p(f)  = 0$. 
Hence,  we can take $\kappa_1(\varepsilon) = \varepsilon$. 

Next, for $r=2$, under the assumption of $e_1$ and $e_2$  in  Theorem~\ref{thm:anyrbetter}
the bound~\eqref{eq:CoPin} implies 
\[
 S_p(f) \ll  p^{8/13 +32\varepsilon/13}  +   \(p^{3/26-46\varepsilon/13}\)^{13/46} p^{89/92}
  \ll p^{1-\varepsilon}
\]
since for $\varepsilon \le 1/9$ we have
\[
\max\left\{8/13 +32\varepsilon/13, \frac{13}{46}  \(3/26-46\varepsilon/13 \) + 89/92\right\} \le 1-\varepsilon. 
\]
It remains to observe that $\min\{31/240, 1 /9\} > 3/92$. 

Finally, under the  the assumption of $e_1$ and $e_2$  in  Theorem~\ref{thm:anyrbetternew}, 
by~\eqref{eq:CoPin}, we have 
\[
 S_p(f) \ll  p^{16/23} +    p^{89/92} \le  p^{89/92} 
=p^{1-3/92}. 
\]
Hence, recalling the definition~\eqref{eq:kappa2}, we see that 
this establishes  the basis of induction for  Theorems~\ref{thm:anyrbetter} and~\ref{thm:anyrbetternew}.

 \subsection{Induction step}\label{sec:induction}
 We now assume that $r \ge 3$ and that the desired result holds for all $k$-sparse polynomials 
 with $k < r$ and establish it for $r$-sparse polynomials.

 Let $\Delta$ be as in~\eqref{eqn:r-1}. 
 By the inductive assumption, we have 
 \begin{align*}
  \Delta & =   \sum_{\substack{\lambda_1, \ldots, \lambda_{r-1}\in \F_p\\
 (\lambda_1, \ldots, \lambda_{r-1}) \ne \mathbf 0}}
\left|\sum_{y\in  \mathbb{F}_p^{*}}  \e_p({\lambda_1 y^{e_1} + \cdots + \lambda_{r-1}y^{e_{r-1}}}) 
\right | ^{2t} \\
& \le  p^{2(t-2)(1- \kappa_{r-1}(\varepsilon))}   \sum_{\substack{\lambda_1, \ldots, \lambda_{r-1}\in \F_p\\
 (\lambda_1, \ldots, \lambda_{r-1}) \ne \mathbf 0}}
\left|\sum_{y\in  \mathbb{F}_p^{*}}  \e_p({\lambda_1 y^{e_1} + \cdots + \lambda_{r-1}y^{e_{r-1}}}) 
\right | ^4\\
& \le  p^{2(t-2)(1- \kappa_{r-1}(\varepsilon))} I_{r-1,2} \le    p^{2(t-2)(1- \kappa_{r-1}(\varepsilon))} I_{2,2}. 
\end{align*}

Hence, by  Lemma~\ref{lem:I22} we see that 
\[
  \Delta \ll   p^{2(t-2)(1- \kappa_{r-1}(\varepsilon)) + 3 - \varepsilon} 
  = p^{2t - 1-  2(t-2)\kappa_{r-1}(\varepsilon)) - \varepsilon}. 
\] 
Then we see from~\eqref{eqn:r-1} that 
 \[
 I_{r-1,t} \ll p^{2t-r+1} +  p^{2t - 1-  2(t-2)\kappa_{r-1}(\varepsilon)) - \varepsilon} 
 \ll  p^{2t-r+1} 
 \]
 provided 
 \begin{equation}\label{eq:t kappa r}
2(t-2)\kappa_{r-1}(\varepsilon) \ge r-2- \varepsilon.
\end{equation}

Hence, under  the condition~\eqref{eq:t kappa r}, by Lemma~\ref{lem:S and I}, taken with $s= 2$, 
we have 
\[
S_p(f)^{4t} \ll p^{4t - 3} \, I_{r-1,2} \leq p^{4t - 3} \, I_{2,2}.
\]
Since  under the conditions of Theorem~\ref{thm:anyrbetter},  the conditions of Lemma~\ref{lem:I22} are 
also satisfied, 
we derive 
\[
S_p(f)^{4t} \ll p^{4t - 4 \varepsilon}.
\]
Taking 
 \[
 t = t_r(\varepsilon), 
 \]
 where $t_r(\varepsilon)$ is given by~\eqref{eq:tr-eps}, 
 to ensure that~\eqref{eq:t kappa r} holds, we derive $S_p(f)  \ll p^{1 -   \varepsilon/t_r(\varepsilon)} = p^{1 -   \kappa_r(\varepsilon)}$, where $ \kappa_r(\varepsilon)$ is as in~\eqref{eq:kappar}.  
 
 Finally, we see that under the conditions of  Theorem~\ref{thm:anyrbetternew} we can apply Lemma~\ref{lem:newI22}
 and check that the above bounds hold with $\varepsilon = 3/92$. 

\section{Proof of Theorem~\ref{thm:coprimeq}}

\subsection{Exponential sums modulo prime powers}

We first recall the following bound on the number of zeros of sparse polynomials in $\F_p$, 
which follows instantly from~\cite[Lemma~7]{CFKLLS}. 

\begin{lem}
\label{lem:SprEq Zeros} Let $p$ be any prime and let  $f\in \F_p[X]$ be be as  in~\eqref{eq: Sparse Poly} and~\eqref{eq: GCD} with $q=p$. 
Let 
\[
D =  \max_{1\le i \ne j \le r} \gcd(e_i - e_j,p-1). 
\]
Then, the equation
\[
f(X) = 0, \qquad x \in \F_p,
\]
has at most $O\(p^{1 - 1/r} D^{1/r}\)$  solutions. 
\end{lem}

We now estimate 
\begin{lem}\label{lem:pm}
Let $p$ be any prime and let  $m\geq 2$ be any integer. Assume  that $f$ is as in~\eqref{eq: Sparse Poly} satisfying~\eqref{eq: GCD} with $q = p$. 
Assume  that for all  $1 \le i \le r$ we have 
\begin{equation}\label{eq: cond pm}
\gcd(e_i, p^m) \le p^{m/3}. \end{equation}
Then, we have   
\[
S_{p^m}(f) \ll  p^{m-1/r} D^{1/r}
\]
where $D$ is as in Lemma~\ref{lem:SprEq Zeros}. 
\end{lem}

\begin{proof} We start as in the proof of~\cite[Proposition~8]{Bourg2}. Let us first set,
\[
n = \rf{m/2}.
\]
Writing
\[
S_{p^m}(f)  = 
\frac{1}{p^n} \sum_{x\in \Z_{p^m}^*} \sum_{1 \leq y \leq p^n}\e_{p^m}(f(x + p^n y)),
\]
and noting that
\[
f(x + p^n y) = f(x) + p^n f'(x)y \pmod{p^m},
\]
we have
\begin{align*}
    |S_{p^m}(f)|&\leq \frac{1}{p^n}\sum_{x\in \Z_{p^m}^*} \left| \sum_{1 \leq y < p^n} \e_{p^{m-n}}(f'(x)y) \right|  \\
    &= \#\left\{x\in \Z_{p^m}^*:~  f'(x) \equiv 0 \pmod{p^{m-n}} \right\} .
\end{align*}
We now write 
\[
f'(x) = \sum_{i=1}^r a_i e_i x^{e_i - 1} = p^u \sum_{i=1}^r b_i x^{e_i - 1},
\]
with some integers $u$ and $b_1, \ldots, b_r$ satisfying $\gcd(b_1, \ldots, b_r,p)=1$.
Since $m \ge 2$, the condition~\eqref{eq: cond pm} implies that $u \le m/3 < m-n$. In particular,
\[
f'(x)\equiv 0 \pmod {p^{m-n}}
\]
implies 
\[
\sum_{i=1}^r b_i x^{e_i - 1}\equiv 0 \pmod {p}.
\]
At this point our argument deviates from that in the proof of~\cite[Proposition~8]{Bourg2} and instead of using 
bounds of exponential sums, we appeal to the bound of Lemma~\ref{lem:SprEq Zeros}, concluding the proof. 
\end{proof} 

\subsection{Concluding the proof}
If 
\[
q = \prod_{p \mid q} p^{m_p}
\] 
is the prime factorisation of $q$, then one can write
\begin{equation}\label{eqn:decomp}
    S_q(f) = \prod_{p\mid q}  S_{p^{m_p}}(\lambda_p f), 
\end{equation}
with some integers $\lambda_p$ such that $\gcd(\lambda_p,p) =1$, see, for example,~\cite[Equation~(12.21)]{IwKow}.

We consider the following system of three independent conditions:
\begin{subequations}\begin{align}
    &\gcd(e_j, p - 1) \le  p^{1/4}, \qquad  1 \le u \le r, \label{eq: cond 2} \\
    &\gcd(e_i - e_j, p - 1) \le p^{3/5}, \qquad  1 \le i\neq j \le r, \label{eq: cond 3}\\
    & \gcd(e_i,e_j, p-1) \le  p^{3/52}, \qquad  1 \le i\neq j \le r. \label{eq: cond 4}
\end{align}
\end{subequations}

Note that the condition~\eqref{eq: cond 3} implies that we have $D \le  p^{3/5}$ in 
Lemmas~\ref{lem:SprEq Zeros} and~\ref{lem:pm}.

Now, we also set 
\[
E = \prod_{1 \le i < j \le r} (e_i - e_j)
\]
and consider the sets of primes
\begin{align*}
   & \cP_1  = \left\{ p:~ \gcd\(p^{m_p}, e_1\ldots e_r \) > p^{m_p / 3} \right \}, \\ 
    &\cP_2  = \left\{ p:~ m_p > 20r^2 \text{ and } 
     E \equiv 0 \pmod{p^{\lfloor m_p / (10r) \rfloor}} 
    \right\}, \\[1em]
    &\cP_3  = \bigl\{ p \not\in \cP_1:~ 
    m_p \le 20r^2 \text{ and } 
    p \text{ does not satisfy} \\
    & \qquad \qquad \qquad \qquad \qquad \qquad \qquad   \text{ one of } \eqref{eq: cond 2}, \eqref{eq: cond 3}, \eqref{eq: cond 4}
    \bigr\}.
\end{align*}
Let us define 
\[
q_i=\prod_{p\in \cP_i} p^{m_p}, \qquad i =1,2,3,
\] 
and follow the arguments as in the proof of in~\cite[Proposition~11]{Bourg2}. 

Note that 
\[
|E| \le d^{r(r-1)/2} < d^{r^2/2}.
\]
Hence for $q_1$ and $q_2$ we have the estimates
\begin{equation}\label{eqn:q1q2}
q_1 < d^{3r} \mand q_2< d^{20 r^3},
\end{equation}
see~\cite[Equations~(3.11) and~(3.12)]{Bourg2}. 

For any $p\in \cP_3$, failing one of the conditions~\eqref{eq: cond 2}, \eqref{eq: cond 3} and~\eqref{eq: cond 4}
means
\[
p^{1/4}\le d, \quad \text{or} \quad    p^{3/5}\le d, \quad \text{or} \quad p^{3/52}\le d. 
\]
Hence, in any case we have $p\le d^{52/3}$. 
Therefore, we trivially  have 
\begin{equation}\label{eqn:q3 prod}
    q_3< \prod_{p \le d^{52/3}} p^{20r^2}.
\end{equation}

By the classical result of Rosser and  Schoenfeld~\cite[Theorem~9]{RoSch}, we see that~\eqref{eqn:q3 prod} implies 
\begin{equation}\label{eqn:q3}
    q_3<  \exp\(1.02 \cdot 20 r^2d ^{52/3}\) \le  \exp\(21 r^2d ^{52/3}\)
\end{equation}  

Next, taking $\cP=\cP_1\cup \cP_2\cup \cP_3$, it follows from~\eqref{eqn:decomp} that
\begin{equation}\label{eqn:afterq}
    S_q(f)\ll q_1q_2q_3\prod_{\substack{p\not\in \cP\\p\mid q}} \left( \sum_{x\in \Z_{p^{m_p}}^*} \e_{p^{m_p}}(\lambda_p f(x)) \right).
\end{equation}
We now follow the argument as in the proof of~\cite[Proposition~11]{Bourg2}, and decompose the product of the sums on the right above as the product of three sums according to, 
\[
m_p=1, \qquad 1<m_p \le 20 r^2, \qquad m_p>20 r^2, 
\] 
and denote these products  $S_1, S_2$, and $S_3$ respectively

To estimate $S_1$,   instead of the conditions of  Theorem~\ref{thm:anyrbetter} corresponding to  $\varepsilon = 3/184$ we impose more stringent 
conditions~\eqref{eq: cond 2}, \eqref{eq: cond 3} and~\eqref{eq: cond 4}, 
which guarantee that the bound of Theorem~\ref{thm:anyrbetter}  holds with 
$ \kappa_r(3/184) \ge  \sigma_r$, as defined in~\eqref{eq:sigma-r},  and we derive 
\begin{equation}\label{eqn:s1}
    S_1\ll Q_1^{1-\sigma_r+o(1)}, \quad \mathrm{where} \  Q_1=\prod_{\substack{p\not \in \cP,~p\mid q\\m_p =1}} p,
\end{equation}
where, here and below,  the term $Q^{o(1)}$ from the classical bound on the divisor function~\cite[Equation~(1.81)]{IwKow}.

Moreover, applying Lemma~\ref{lem:pm} for $m_p \ge 2$, 
Since for  $p\not \in \cP$ with $m_p \ge 2$, 
 the condition~\eqref{eq: cond 3} that we have $D \le p^{3/5}$. Hence
 \begin{equation}\label{eqn:s2}
    S_2\ll Q_2^{1-1/(50r^3)+o(1)}, \quad \mathrm{where} \ Q_2=\prod_{\substack{p\not \in \cP,~p\mid q\\ 2 \le m_p\le 20 r^2}} p^{m_p}.
\end{equation}

Finally to estimate $S_3$, since we run over $m_p>20r^2$ while $p\not\in \cP_2$, we  apply~\cite[Proposition~7]{Bourg2}, and get
\begin{equation}\label{eqn:s3}
    |S_3|\le Q_3^{1-1/(10r^2)+o(1)}, \quad \mathrm{where} \ Q_3=\prod_{\substack{p\not \in \cP,~p\mid q\\ m_p>20 r^2}} p^{m_p}.
\end{equation}

For any integer $r \ge 1$ we have 
\[
\sigma_r \le \frac{1}{50r^3}  . 
\] 
Combining~\eqref{eqn:afterq}, \eqref{eqn:s1},  \eqref{eqn:s2}, and~\eqref{eqn:s3}, we have
\[
S_q(f) \ll q_1q_2q_3 \(Q_1Q_2Q_3\)^{1-\sigma_r+o(1)} = q^{1-\sigma_r+o(1)} \(q_1q_2q_3\)^{\sigma_r}.
\]
Combining~\eqref{eqn:q1q2} and~\eqref{eqn:q3}, we see that 
\[
q_1q_2q_3  \le d^{3r+20 r^3} \exp\( 21 r^2 d ^{52/3}\)     \le 
 \exp\(r^3d^{18}\), 
\]
concluding the proof.

\section{Proof of Theorem~\ref{thm:discr PowGen}}

\subsection{Discrepancy   and exponential sums}
\label{sec:discrepancy}

Let  $D(\Gamma)$ denote the discrepancy 
of  a sequence $\Gamma=(\bgamma_n)_{n= 0}^N$ of $N$ points 
in an $s$-dimensional unit cube $[0,1]^s$ as defined in Section~\ref{sec:mulf distr}.

One of the basic tools used to study uniformity of
distribution is the celebrated
 Koksma--Sz\"usz inequality~\cite{Kok,Sz} 
(see also~\cite[Theorem~1.21]{DrTi}), which generalises the Erd\H{o}s--Tur\'an inequality to arbitrary dimensions. Like the Erd\H{o}s--Tur\'an inequality, the Koksma--Sz\"usz inequality gives a bound on the discrepancy via certain
exponential sums.

\begin{lem}\label{lem:K-S}
For any integer $A \ge 1$, we have
\[
D(\Gamma) \ll \frac{1}{A}
+\frac{1}{N}\sum_{\substack{\va \in \mathbb{Z}^s\\ 0<\|\va\|\le A}}\frac{1}{r(\va)}
\left| \sum_{n=1}^N \e{(\langle \va, \bgamma_n\rangle)}\right|,
\]
where $ \langle \va, \vb\rangle$ denotes the inner product in $\R^s$, 
\[
\|\va\|=  \max_{j=1, \ldots, s}|a_j|, \qquad
r(\va) =  \prod_{j=1}^s \(|a_j| + 1\), \qquad \e(z) =  \exp\(2\pi i z\). 
\]
\end{lem}   

 \subsection{Poisson summation} 
 
 We formulate  the following result, which is well known and with various modifications
 has been used in a large number of works. However,  since finding a precise 
 general reference 
 is not easy, we present it here with a full proof, which follows closely the argument 
 of Friedlander and    Iwaniec~\cite[Section~3]{FrIw}. 
 
Let, as before $\e(z) =  \exp\(2\pi i z\)$. 

 \begin{lem}\label{lem:Poisson} Let $x_n$, $n \in \Z$ be a both-sided infinite sequence of 
complex numbers, indexed by integers. Then for any  integers $N \ge K \ge 1$, there is a real $\alpha$, 
 such that 
\[
\sum_{n=1}^N  x_n \ll \frac{\log N}{K} 
\sum_{n=-N}^N \left| \sum_{k=1}^K x_{n+k} \e(\alpha k)\right|. 
\]
\end{lem}

\begin{proof}  
As in~\cite[Section~3]{FrIw},  consider the function 
\[
f(x) = \min\{x, 1, N+1-x\}
\]
in the interval $x \in [0, N+1]$ and set $f(x) = 0$ elsewhere
(that is, the graph of $f(x)$ looks like a trapezoid). 
Since $\Z$ is invariant with respect to shifting by $k\in \Z$,  we see that 
\[
\sum_{n=1}^N  x_n = \sum_{n\in \Z}  f(n) x_n  = \sum_{n\in \Z}  f(k+n) x_{k+n} .
\]
Hence, summing over $k=1, \ldots, K$ and changing the order of summation, yields
\[
\sum_{n=1}^N  x_n = \frac{1}{K}  \sum_{n\in \Z} \sum_{k = 1}^K f(k+n) x_{k+n} 
=  \frac{1}{K}  \sum_{n=-N}^N \sum_{k = 1}^K f(k+n) x_{k+n} .
\]

Let 
\[
g(y) = \int_{-\infty}^\infty f(x) \e(-xy) dx
\]
be the Fourier transform of $f(x)$. 
Applying Poisson summation to  the sum  over  $n$ gives 
\[
\sum_{n=1}^N  x_n  
=  \frac{1}{K}  \int_{-\infty}^\infty g(y) \sum_{n=-N}^N \e(ny) \sum_{k = 1}^K e(ky) x_{k+n} dy. 
\]
Therefore 
\begin{align*}
\sum_{n=1}^N  x_n  
& \ll \frac{1}{K}  \int_{-\infty}^\infty |g(y)|  \sum_{n=-N}^N \left| \sum_{k = 1}^K \e(ky) x_{k+n}\right| dy\\ 
& \ll  \frac{1}{K}  \sum_{n=-N}^N \left| \sum_{k = 1}^K \e(\alpha k) x_{k+n}\right| \int_{-\infty}^\infty |g(y)|  dy  , 
\end{align*}
for some real $\alpha$ for which the maximum is above double sum is achieved. 
Note that this sum is periodic with period 1 and a continuous function of $y$, so such $\alpha$. 
We now recall the bound
\[
\int_{-\infty}^\infty |g(y)| \ll \log N, 
\]
established in~\cite[Section~3]{FrIw}, and conclude the proof.
\end{proof}

\subsection{Multiplicative orders modulo divisors}

There is no doubt that the following elementary statement is well-known, it 
is also given in~\cite[Lemma~6]{OstShp}.  

\begin{lem}\label{lem:MultOrd}
Let $e$ be of multiplicative order $\tau_q$ modulo a positive integer $q$ with $\gcd(e,q)=1$ 
and of multiplicative order $\tau_r$  modulo a positive divisor $r \mid q$.
Then 
\[
\tau_r \ge \frac{r}{q} \tau_q.
\]
\end{lem}  
 
The following lower bound follows immediately from  Lemma~\ref{lem:MultOrd},
and is also given as~\cite[Lemma~7]{OstShp}. 

\begin{lem}\label{lem:Large GCD}
Let $e$ be of multiplicative order $\tau_q$ modulo a positive integer $q$ with $\gcd(e,q)=1$.
Then for any positive integer $h$, 
\[
 \gcd(e^h-1, q)\le hq/\tau_q. 
\]
\end{lem}  

Next, we apply Lemma~\ref{lem:MultOrd} to derive the following variant of~\cite[Lemma~7]{OstShp}.

\begin{lem}\label{lem:MultOrd and GCD}
Let $e$ be of multiplicative order $\tau_q$ modulo a positive integer $q$ with $\gcd(e,q)=1$ 
For a fixed integer $m\ge 0$ there at most $qK/(r \tau_q) +1$ positive integers $k\le K$, 
with 
\[
\gcd(e^k-e^m, q) = r. 
\]
\end{lem}  

\begin{proof} Clearly $k \equiv m \pmod {\tau_r}$, where  $\tau_r$ is the  multiplicative order
of  $e$ modulo   $r$  Hence, there are at most $K/\tau_r  +1$ such positive integers $k\le K$.
Recalling Lemma~\ref{lem:MultOrd}, we derive the desired result. \end{proof}

\subsection{Bounding exponential sums} 

In order to use  Lemma~\ref{lem:K-S} we need to estimate the exponential sums
\[
S_{\va}(N) =  \sum_{n=1}^{N} \ep\(a_1u_{n}+\cdots + a_su_{n+s-1}\), 
\]
with $\va = (a_1, \ldots, a_s) \in \F_p^s$, 
for which we prove that 
\begin{equation}\label{eq: Bound SaN}
|S_{\va}(N)| \ll  \(\sqrt{NT/\tau}  p^{7/46}  + N^{1/2} p^{1/2 -  0.5  \rho_{2s}}\)  p^{o(1)}
\end{equation} 
provided that $\va$ is a non-zero vector. 

We see from~\eqref{eq:Pow}  that 
 \[
 u_n \equiv \vartheta^{e^n} \pmod p. 
 \]
 Since $\gcd(e,T) = 1$,  this allows us to extend the definition of $ u_n$ 
 for all $n \in \Z$.
 
By Lemma~\ref{lem:Poisson},  for some real $\alpha$ we have   
 \[
 S_{\va}(N) \ll \frac{\log p}{K}  \sum_{n=-N}^{N} \left|
 \sum_{k=1}^K \ep\(a_1u_{n+k}+\cdots + a_su_{n+k+s-1}\) \e\(\alpha k\)\right|.
 \]

In fact we simply choose $K = N$ and by  the Cauchy inequality, we have 
\begin{align*}
 S_{\va}(N)^2 & \ll \frac{(\log p)^2}{N}  \sum_{n=-N}^{N}\left|
 \sum_{k=1}^N \ep\(a_1u_{n+k}+\cdots + a_su_{n+k+s-1}\) \e\(\alpha k\)\right|^2\\
& = \frac{(\log p)^2}{N}  \sum_{n=-N}^{N}\left|
 \sum_{k=1}^N 
\ep\(\sum_{j=1}^{s}a_j \vartheta^{e^{n+k+j}} \) \e\(\alpha k\) \right|^2. 
\end{align*}
Since $N \le \tau$, all elements $ \vartheta^{e^n}$, $n =1, \ldots, N$, are distinct elements of $\F_p$.
Hence in the sequence  $ \vartheta^{e^n}$, $n =0, \pm 1, \ldots, \pm N$, each  
element of $\F_p$ occurs at most $3$ times. 
Therefore 
\[
 S_{\va}(N)^2  \ll  \frac{(\log p)^2}{N}  \sum_{x \in \F_p^*} \left|
 \sum_{k=1}^N 
\ep\(\sum_{j=1}^{s}a_j x^{e^{k+j}}\)  \e\(\alpha k\) \right|^2
 \]
and after squaring out and opening up, we arrive to 
 \begin{equation}\label{eq: sum S}
 S_{\va}(N)^2  \ll  \frac{ (\log p)^2}{N}    \sum_{k,m =1}^N  \left|  \sum_{x \in \F_p^*}
\ep\(\sum_{j=1}^{s}a_j \(x^{e^{k+j}}-x^{e^{m+j}}\) \) \right|. 
\end{equation}

 In order to apply Theorem~\ref{thm:anyrbetternew} we need to control the 
 greatest common divisors for the differences   of the  exponents 
 of the polynomials in the inner sum in~\eqref{eq: sum S}.

 First we observe  that $\tau$ is the multiplicative order of $e$ modulo $T$. 
 We can also assume that 
 \[
 \frac{s T}{\tau} \le  p^{16/23}
 \]
 as otherwise the result is trivial. 
 Hence, for $j > i$, by Lemma~\ref{lem:Large GCD}
\begin{align*}
\gcd\(e^{k+j}- e^{k+i}, p-1\) & = \gcd\(e^{j-i}-1, T\) \\
&  \le \frac{s T}{\tau} \le  p^{16/23}. 
\end{align*}

Next, we fix a divisor $R \mid T$ and observe that by Lemma~\ref{lem:MultOrd and GCD}
the number of pairs $(k,m)$, $1\le k,m  \le N$, with 
 \[
 \max_{i,j =1, \ldots, s} \gcd\(e^{k+i}- e^{m+j}, p-1\)  =  R
\]
is bounded by $O\(N \(NT/(R \tau) +1\)\)$. 
Hence, using  this for all divisors  $R \mid T$ with 
$R >  p^{16/23}$ using the classical 
bound on the divisor function (see, for example,~\cite[Equation~(1.81)]{IwKow}), we see that the
number $U$ of  pairs $(k,m)$, $1\le k,m  \le K$, with 
 \[
 \max_{i,j =1, \ldots, s} \gcd\(e^{k+i}- e^{m+j}, p-1\)   >  p^{16/23}
\]
can be  bounded as 
\[
U  \ll N   \(\frac{NT}{ p^{16/23} \tau} +1\) p^{o(1)}.
\]
For such pairs we estimate the inner sum over $x \in \F_p$ in~\eqref{eq: sum S} trivially as $p$ 
and, since  $\va$ is a non-zero vector, we can use the bound of Theorem~\ref{thm:anyrbetternew} 
for the remaining pairs. 

Thus, we derive
\begin{align*}
 S_{\va}(N)^2 &  \ll  \frac{ (\log p)^2}{N}   \(U p + N^2 p^{1 -  \rho_{2s}}\)\\
& \le \( \frac{NT p^{7/23}}{\tau} +p + N p^{1 -  \rho_{2s}}\)  p^{o(1)}.
\end{align*}
Clearly the above bound is trivial if  $p^{1 -  \rho_{2s}} \ge N$, while otherwise 
\[
p \le  N p^{1 -  \rho_{2s}}.
\]
Hence the above bound can be simplified as 
\[
|S_{\va}(N)|^2  \le \( \frac{NT p^{7/23}}{\tau}  + N p^{1 -  \rho_{2s}}\)  p^{o(1)},
\]
and yields~\eqref{eq: Bound SaN}. 

In turn, the bound~\eqref{eq: Bound SaN}, together with Lemma~\ref{lem:K-S}  taken with $A = (p-1)/2$,
yields the bound
 \begin{equation}\label{eq: prelim D}
D_{e, \vartheta,s}(p,N) \le  N^{-1/2}  \(\sqrt{T/\tau}  p^{7/46}  + p^{1/2 -  0.5  \rho_{2s}}\)   p^{o(1)}.
\end{equation}

Next we observe that the above bound is trivial for $N < p^{1- \rho_{2s}}$. Hence we can assume that $\tau \ge N \ge p^{1- \rho_{2s}}$. In this case 
\[
\sqrt{T/\tau}  p^{7/46}  \le \sqrt{p/\tau}  p^{7/46} \le p^{7/46 + 0.5   \rho_{2s}}
\le  p^{1/2 -  0.5  \rho_{2s}}
\]
since $ \rho_{2s} \le  \rho_{2} = 3/92$. 

Thus, the first term in~\eqref{eq: prelim D} never dominates, and the desired result follows.

\section{Multidimensional Generalisations}

As in~\cite{OstShp} one can consider various generalisations of the power generator which arise 
from  the dynamical system generated by iterations of multivariable polynomials. 

Let $F_1, \ldots, F_m \in \mathbb{F}_p[x_1, \ldots, x_m]$ be $m$ polynomials in $m$ variables over $\mathbb{F}_p$. For each $i = 1, \ldots, m$ we define the $k$-th iteration of the polynomial $F_i$ by the recurrence relation
\[
F_i^{(0)} = X_i, \quad F_i^{(k)} = F_i \left( F_1^{(k-1)}, \ldots, F_m^{(k-1)} \right), \quad k = 1, 2, \ldots.
\]
More precisely, we define the vectors $u_n = (u_{n,1}, \ldots, u_{n,m}) \in \mathbb{F}_p^m$ by the recurrence relation
\[
u_{n+1,i} = F_i(u_{n,1}, \ldots, u_{n,m}), \quad n = 0, 1, \ldots, \quad i = 1, \ldots, m,
\]
with some initial vector $\vec = (u_{0,1}, \ldots, u_{0,m}) \in \mathbb{F}_p^m$. The sequence $(\vec{u}_n)$ is eventually periodic with some period $\tau \leq p^m$.

Then, we denote $D_{s,\vec{F},\vec{u}}(p,N)$ 
the discrepancy of the $N$ points 
\[
\(\frac{u_{n,1}}{p},\ldots, \frac{u_{n,m}}{p} \ldots, \frac{u_{n+s-1,1}}{p},\ldots,\frac{u_{n+s-1,m}}{p}\),\quad n =1, \ldots, N, 
\]
inside $[0,1]^{ms}$.

Following~\cite{OstShp}, we consider the following systems of polynomials.
\subsubsection*{Type~1} As in~\cite[Section~3.1]{OstShp}, we consider polynomials 
\begin{align*} 
&F_1 = (X_1 - h_1)^{e_1} G_1 + h_1, \\
&\vdots   \qquad \qquad \qquad \quad \vdots\\
&F_{m-1} = (X_{m-1} - h_{m-1})^{e_{m-1}} G_{m-1} + h_{m-1}, \\ 
&F_m = g_m (X_m - h_m)^{e_m} + h_m,
\end{align*}
where $G_i \in \mathbb{F}_p[X_{i+1}, \ldots, X_m]$ for $i = 1, \ldots, m-1$, and $g_m, h_i \in \mathbb{F}_p$ for $i = 1, \ldots, m$, such that none of the $G_i$ has zeroes over $\mathbb{F}_p$.
\subsubsection*{Type~2} The following generalisation of the power generator has been suggested in~\cite[Section~4]{OstShp}. Let 
$F_1 = X_1^e L_1 + L_0$, where $L_0, L_1$ and the remaining 
polynomials $F_2, \ldots, F_m$ are all non-constant linear polynomials in $\F_p[X_2, \ldots, X_m]$.

\subsubsection*{Type~3}
For some integral    triangular 
matrix $ \left( e_{i,j} \right)_{i,j=1}^m \in \mathbb{Z}^{m \times m}$ of exponents, we consider the system of monomials  
\[
F_j= X_1^{e_{j,1}} \cdots X_m^{e_{j,m}}, \qquad j = 1, \ldots, m. 
\]
which provides the most natural generalisation of the power generator but also inherits its negative aspects such as homogeneity, 
see~\cite[Sections~1.2 and~4]{OstShp}.

The argument of~\cite{OstShp} allows to get an explicit estimate the discrepancy of the corresponding sequences of vectors only for $s=1$, 
see, for example,~\cite[Theorem~3.2]{FrSh} or~\cite[Corollary~9]{OstShp}, and appeal to the result of  Bourgain~\cite[Theorem~1]{Bourg1} (that is, to~\eqref{eq:Bourg}), 
otherwise. 

One can use Theorem~\ref{thm:anyrbetternew} and argue as in~\cite{OstShp} to get explicit estimates for the discrepancy $D_{s,\vec{F},\vec{u}}(p,N)$
for any fixed $s \ge 1$, 
via  the Koksma--Sz\"usz  inequality given in Lemma~\ref{lem:K-S}.

It is certainly natural to try to extend the above 
constructions to other systems of polynomials involving more monomials and then use
Theorem~\ref{thm:anyrbetternew}.  Unfortunately 
the sparsity of iterations of such polynomials grows exponentially with the number of 
iterations. So we pose it as an open question to find other polynomial systems 
to which our new bound of exponential sums applies.

\section*{Acknowledgements}

The authors would like to thank  Arne Winterhof for very helpful comments and 
suggestions.

During the preparation of this work, the  authors   were partially supported by the
Australian Research Council Grant DP230100530.

\end{document}